\documentclass{article}
\usepackage[utf8]{inputenc}
\usepackage[T1]{fontenc}
\usepackage{amsmath}
\usepackage{amssymb}
\usepackage{amsthm}
\usepackage{enumerate}
\usepackage{mathrsfs}
\usepackage{scrextend}
\usepackage{mathabx}
\usepackage{mathtools}
\usepackage[hyperfootnotes=false]{hyperref}
\usepackage{cleveref}
\usepackage[multiple]{footmisc}
\usepackage[a4paper, total={6in, 9in}]{geometry}
\usepackage{scalerel,stackengine}

\newcommand{\N}{\mathbb{N}}
\newcommand{\Z}{\mathbb{Z}}
\newcommand{\Q}{\mathbb{Q}}
\newcommand{\R}{\mathbb{R}}
\newcommand{\C}{\mathbb{C}}
\newcommand{\cX}{\mathcal{X}}

\renewcommand{\P}{\mathbb{P}}

\DeclareMathOperator*{\E}{\mathbb{E}}
\allowdisplaybreaks

\stackMath
\newcommand\reallywidehat[1]{%
\savestack{\tmpbox}{\stretchto{%
  \scaleto{%
    \scalerel*[\widthof{\ensuremath{#1}}]{\kern-.6pt\bigwedge\kern-.6pt}%
    {\rule[-\textheight/2]{1ex}{\textheight}}
  }{\textheight}%
}{0.5ex}}%
\stackon[1pt]{#1}{\tmpbox}%
}
\theoremstyle{plain}
\newtheorem{theorem}{Theorem}
\newtheorem{lemma}[theorem]{Lemma}
\newtheorem{proposition}[theorem]{Proposition}
\newtheorem{corollary}[theorem]{Corollary}
\newtheorem{problem}[theorem]{Problem}

\theoremstyle{definition}

\newtheorem{remark}[theorem]{Remark}

\numberwithin{theorem}{section}

\title{Limiting spectral laws for sparse random circulant matrices}
\author{Adrian Beker\footnote{University of Zagreb, Faculty of Science, Department of Mathematics, Zagreb,
Croatia.\\ Email: \nolinkurl{adrian.beker@math.hr}}}
\date{\today}

\begin{document}

\maketitle

\begin{abstract}
Fix a positive integer $d$ and let $(G_n)_{n\geq1}$ be a sequence of finite abelian groups with orders tending to infinity. For each $n \geq 1$, let $C_n$ be a uniformly random $G_n$-circulant matrix with entries in $\{0,1\}$ and exactly $d$ ones in each row/column. We show that the empirical spectral distribution of $C_n$ converges weakly in expectation to a probability measure $\mu$ on $\C$ if and only if the distribution of the order of a uniform random element of $G_n$ converges weakly to a probability measure $\rho$ on $\N^*$, the one-point compactification of the natural numbers. Furthermore, we show that convergence in expectation can be strengthened to convergence in probability if and only if $\rho$ is a Dirac mass $\delta_m$. In this case, $\mu$ is the $d$-fold convolution of the uniform distribution on the $m$-th roots of unity if $m\in\N$ or the unit circle if $m = \infty$. We also establish that, under further natural assumptions, the determinant of $C_n$ is $\pm\exp((c_{m,d}+o(1))|G_n|)$ with high probability, where $c_{m,d}$ is a constant depending only on $m$ and $d$.
\end{abstract}

\section{Introduction}\label{sec:intro}

\noindent\textbf{History and previous work.} Given an $n\times n$ complex matrix $M$, we denote by $(\lambda_k(M))_{1\leq k\leq n}$ and $(\sigma_k(M))_{1\leq k\leq n}$ its eigenvalues and singular values respectively. The spectral and singular value distribution of $M$ are given by
\[\mu_M \vcentcolon= \frac{1}{n}\sum_{k=1}^{n}\delta_{\lambda_k(M)}, \quad \nu_M \vcentcolon= \frac{1}{n}\sum_{k=1}^{n}\delta_{\sigma_k(M)}.\] 
The study of the limiting behaviour of these measures for various sequences of random matrices is a central topic in random matrix theory. It goes back to the seminal works of Wigner \cite{wigner} and Marchenko--Pastur \cite{marchenko-pastur}, which establish the well-known semi-circular and quarter-circular laws. These results deal with spectra of symmetric or Hermitian random matrices. In the non-Hermitian case, which is usually more challenging, a natural model is given by matrices with independent entries drawn from the same fixed distribution $\xi$. A famous conjecture in this area stated that, whenever $\xi$ has mean zero and unit variance, the empirical spectral distribution (\emph{ESD} from now on) of suitably scaled such matrices converges to a universal limiting law, namely the uniform distribution on the unit disk in the complex plane. Closing a long line of research, this conjecture was resolved in full by Tao and Vu in their celebrated paper \cite{tao-vu-circular}. For a detailed account of the history of this problem, we direct the reader to the comprehensive survey \cite{bordenave-chafai}.

Although this area had seen a wealth of new results and techniques in the past twenty years, some important problems remained open. Recently, one such problem was solved in breakthrough work of Sah, Sahasrabudhe and Sawhney \cite{sah-sahasrabudhe-sawhney}. They proved convergence of ESDs of sparse i.i.d.\ matrices, namely those with entries in $\{0,1\}$ and having a fixed expected number of ones in each row/column. A related problem, which is still open, is the following (see \cite{chafai} or \cite[Problem 5]{tikhomirov}):

\begin{problem}
\label{prob:digraphs}
Fix an integer $d \geq 3$ and for each $n \geq d$ let $A_n$ be a matrix draw uniformly at random from all $n \times n$ matrices with entries in $\{0,1\}$ and exactly $d$ of ones in each row and column. Show that $(\mu_{A_n})_{n\geq d}$ converges weakly in probability to the \emph{oriented Kesten--McKay law}, the $d$-fold free convolution of the uniform distribution on the unit circle.
\end{problem}

Changing gears slightly, we turn our attention to \emph{patterned} random matrices. Such matrices possess additional structure reflected by linear-algebraic relations between their entries. At the same time, this results in a limited amount of stochastic independence in comparison to the more classical models discussed previously. Of particular importance in the literature are Hankel, Toeplitz and circulant matrices \cite{bose-gangopadhyay-sen, bose-hazra-saha, bose-mitra, bose-saha, bose-sen, bose-subhra-saha, bryc-dembo-jiang, meckes-cyclic, meckes-abelian, adamczak, banerjee-bose, barrera-manrique, latala-swiatkowski}; our focus will be on the latter class. Following Diaconis \cite[\S3E]{diaconis}, for a finite group $G$ we define a \emph{$G$-circulant matrix} to be a matrix $(A_{x,y})_{x,y\in G}$ whose rows and columns are indexed by elements of $G$ and whose $(x,y)$-entry is given by $A_{x,y} = a(xy^{-1})$ for some fixed function $a \colon G \to \C$. We think of the first row of $A$ as being indexed by the identity element of $G$. Hence, we will sometimes refer to $A$ as having first row $\widetilde{a}$, where we define
\begin{equation}\label{eq:tilde}
    \widetilde{a} \colon G \to \C, \quad y \mapsto a(y^{-1}).
\end{equation}
In the case when $G = \Z/n\Z$, we thus recover the classical notion of an $n \times n$ circulant matrix. In general, $A$ is the matrix of the convolution operator $f \mapsto a*f$ with respect to the standard basis of $\C^G$. Operators of this kind feature in a wide variety of areas of both pure and applied mathematics, making circulant matrices a very natural object to study.

The study of non-Hermitian random circulant matrices began only fairly recently and has its inception in the work of Meckes \cite{meckes-cyclic}. In that paper, Meckes establishes an analogue of the circular law for classical circulant matrices with first row consisting of i.i.d.\ random variables with mean zero and unit variance. He proves that the ESD of such matrices converges weakly in probability to the standard complex Gaussian distribution. In the follow-up paper \cite{meckes-abelian}, Meckes extended this result to $G$-circulant matrices for arbitrary finite abelian groups $G$. In this more general setting, a mixture of real and complex Gaussian distributions can arise as the limiting law, and this happens in the presence of an asymptotically non-neglible proportion of elements of order $2$ in $G$. When $G$ is allowed to be non-abelian, Adamczak \cite{adamczak} established universality of the limiting singular value distribution and, in the case of independent Gaussian entries, the limiting spectral distribution.

\smallskip

\noindent\textbf{Main results.} In the present paper, we consider the following circulant analogue of Problem \ref{prob:digraphs}. Let $d$ be a positive integer and let $(G_n)_{n\geq1}$ be a sequence of finite abelian groups such that $|G_n| \to \infty$ as $n \to \infty$. This data can be thought of as being fixed throughout the paper. We let $C_n$ be a matrix drawn uniformly at random from $\mathrm{Mat}_d(G_n)$, the set of all $G_n$-circulant matrices with entries in $\{0,1\}$ and exactly $d$ ones in each row/column. We are interested in the limiting behaviour of the ESDs $(\mu_{C_n})_{n\geq1}$ in the sense of weak convergence in expectation and in probability.\footnote{For a precise definition of these terms, see Section \ref{subsec:prob_weak_conv}.}

To state our results, we require some notation. If $m$ is a positive integer, we let $R_m$ be the set of $m$-th roots of unity in $\C$. We define
\[\eta_m \vcentcolon= \frac{1}{m}\sum_{k=1}^{m}\delta_{e^{2\pi ki/m}}\]
to be the uniform distribution on $R_m$ and $\eta_{\infty}$ to be the uniform distribution on the unit circle in $\C$. We use the additive-combinatorial notation $dR_m$ to denote the $d$-fold sumset of $R_m$, that is, the set of all possible sums of $d$ elements of $R_m$. Similarly, $\eta_m^{*d}$ denotes the $d$-fold convolution of $\eta_m$, where we allow $m$ to be equal to $\infty$. Given a finite abelian group $G$ and an element $x \in G$, we denote by $\mathrm{ord}_G(x)$ its order in $G$ and define
\[\rho_G \vcentcolon= \frac{1}{|G|}\sum_{x\in G}\delta_{\mathrm{ord}_G(x)}.\]
For the purposes of taking weak limits, it is convenient to regard $\rho_G$ as a probability measure on $\N^*$, the one-point compactification of $\N$. Finally, recall that the \emph{exponent} of $G$, written $\exp(G)$, is the least positive integer $r$ such that $rx = 0$ for all $x \in G$.

Our first result characterises when the sequence $(\mu_{C_n})_{n\geq1}$ converges weakly in expectation and identifies the limiting law if convergence holds.

\begin{theorem}
\label{thm:conv_in_exp}
Suppose that $(\mu_{C_n})_{n\geq1}$ converges weakly in expectation. Then there exists a probability measure $\rho$ on $\N^*$ such that $\rho_{G_n} \to \rho$ weakly. Conversely, if $\rho_{G_n} \to \rho$ weakly, then $\mu_{C_n} \to \mu$ weakly in expectation, where
\begin{equation}\label{eq:limit_in_exp}
    \mu = \sum_{m\in\N^*}\rho(\{m\})\eta_m^{*d}.
\end{equation}
\end{theorem}

The second result, which can be regarded as the main contribution of the paper, tells us precisely when one can upgrade weak convergence in expectation to the stronger notion of weak convergence in probability.

\begin{theorem}
\label{thm:conv_in_prob}
Suppose that $(\mu_{C_n})_{n\geq1}$ converges weakly in probability. Then there exists $m \in \N^*$ such that $\rho_{G_n} \to \delta_m$ weakly. Conversely, if $\rho_{G_n} \to \delta_m$ weakly, then $\mu_{C_n} \to \eta_m^{*d}$ weakly in probability.
\end{theorem}

\begin{remark}
\label{rem:exp_but_not_prob}
It can happen that $(\mu_{C_n})_{n\geq1}$ converges weakly in expectation, but not in probability. For instance, if $G_n = (\Z/2\Z)\oplus(\Z/3\Z)^n$, then $\rho_{G_n} \to \frac{1}{2}\delta_3 + \frac{1}{2}\delta_6$ weakly, which is not a Dirac mass.
\end{remark}

\begin{remark}
\label{rem:examples}  
It is useful to keep in mind the following two prototypical examples for the sequence $(G_n)_{n\geq1}$. First, if $G_n = \Z/n\Z$, it is easy to see that $\rho_{G_n} \to \delta_{\infty}$ weakly, so $(\mu_{C_n})_{n\geq1}$ converges weakly in probability to $\eta_{\infty}^{*d}$, the $d$-fold convolution of the uniform distribution on the unit circle. We thus obtain a commutative counterpart of the limiting law predicted by Problem \ref{prob:digraphs}. On the other hand, in the case when $G_n = (\Z/m\Z)^n$ with $m \in \N$ fixed, we have that $\rho_{G_n} \to \delta_m$ weakly, so $\mu_{C_n} \to \eta_m^{*d}$ weakly in probability.
\end{remark}

Our final result states that, provided no matrix in $\mathrm{Mat}_d(G_n)$ is singular, asymptotically almost all matrices in $\mathrm{Mat}_d(G_n)$ have determinant $\pm\exp((c_{m,d}+o(1))|G_n|)$. A similar result for circulant matrices whose first row consists of independent Rademacher random variables is established in forthcoming work of Eberhard and \'O Cath\'ain \cite{eberhard-ocathain}.

\begin{theorem}
\label{thm:asymp_det}
Suppose that $0 \not\in dR_{\exp(G_n)}$ for all $n \geq 1$ and $\rho_{G_n} \to \delta_m$ weakly, where $m \in \N^*$. Then
\[\frac{1}{|G_n|}\log|\det(C_n)| \to c_{m,d}\] 
in probability, where
\[c_{m,d} = \int_{\C}\log|z|\,d\eta_m^{*d}(z)\]
is a non-negative real number. In other words, for any $\varepsilon, \delta > 0$ there exists $n_0 \geq 1$ such that for all $n \geq n_0$ we have
\[\P\Bigl(\exp\bigl((c_{m,d}-\varepsilon)|G_n|\bigr) \leq |\det(C_n)| \leq \exp\bigl((c_{m,d}+\varepsilon)|G_n|\bigr)\Bigr) \geq 1-\delta.\]
\end{theorem}

In the rest of the paper, we work in a model which is slightly different to the one stated above. More precisely, we let
\begin{equation}\label{eq:iid_sum}
    S_n \vcentcolon= \sum_{j=1}^{d}1_{\{X_{n,j}\}},
\end{equation} 
where $(X_{n,j})_{1\leq j \leq d}$ are independent uniform random elements of $G_n$. We redefine $C_n$ to be the $G_n$-circulant matrix with first row $\widetilde{S_n}$, where $\widetilde{S_n}$ is defined as in \eqref{eq:tilde}. In other words, our new model is obtained by sampling $d$ elements of $G_n$ with replacement. It is contiguous with the original model in a rather strong sense. Indeed, by conditioning on the event 
\begin{equation}\label{eq:diff_elements}
    \{\forall i,j \in [d]\quad (i \neq j \implies X_{n,i} \neq X_{n,j})\}
\end{equation}
we recover the previous model. Since $d$ is fixed and the order of $G_n$ tends to infinity, this event has probability tending to $1$, so all questions of convergence remain unaffected.\footnote{This is actually the only place in the paper where the assumption about the growing orders of $(G_n)_{n\geq1}$ is used.}

\smallskip

\noindent\textbf{Methodology.} In the spectral theory of Hermitian random matrices, a common line of attack proceeds via the method of moments (see e.g.\ \cite[\S3]{fleermann-kirsch} or \cite[\S2]{tao-matrices}). However, in the case of general non-Hermitian matrices, the method of moments cannot be used to control the behaviour of the spectrum. This has to do with the fact that probability measures on $\C$, in contrast to those on $\R$, stand no chance of being determined by their moments (see \cite[\S5.3]{bordenave-chafai} or \cite[\S2.8.2]{tao-matrices}). Nonetheless, in certain special circumstances, there are ways round this obstacle. For instance, Meckes \cite{meckes-cyclic} exploits the fact that the eigenvalues of circulant matrices are linear functionals of their entries via a multidimensional central limit theorem. Such a result, however, is not available in the sparse regime considered in our paper. 

Fortunately, there are well-established methods that are robust enough to deal with our setting. Our proof of Theorem \ref{thm:conv_in_prob} relies on the Hermitisation approach pioneered by Girko \cite{girko} and subsequently made rigorous in the series of works on the circular law (see e.g.\ \cite[\S4.1]{bordenave-chafai} or \cite[\S2.8.3]{tao-matrices}). At its heart, this method uses the logarithmic potential to relate the distribution of the eigenvalues to that of the singular values. The latter is supported on the real line, so one can apply the usual method of moments to it. However, since the logarithm is unbounded at $0$ and $\infty$, a price is paid in the process of passing to Hermitian matrices. To deal with this issue, one has to justify that the logarithm is uniformly integrable with respect to (shifted) singular value measures. It is exactly this step that usually turns out to be the most difficult. Its successful completion often requires the development of novel ways of controlling the behaviour of extreme singular values, e.g.\ \cite{tao-vu-moment} or \cite{sah-sahasrabudhe-sawhney}.

Even though our matrices are not Hermitian, there is an important property that distinguishes them from generic non-Hermitian matrices, namely that of being \emph{normal}. In particular, this means that their singular values are the absolute values of the eigenvalues. In general, this tight relationship brings the study of spectral and singular value distributions much closer together. Furthermore, our matrices are simultaneously diagonalised by the Fourier transform, which results in explicit expressions for eigenvalues in terms of the entries. This makes reasoning about singular values significantly more tractable.

In view of the preceding paragraph, we should note that Hermitisation is not strictly necessary in order to establish Theorem \ref{thm:conv_in_prob}. Indeed, there is a variant of the moment method, hinted at in e.g.\ \cite[p.\ 49]{bordenave-chafai} or \cite[p.\ 249]{tao-matrices}, that works for normal matrices $M$. Briefly, one considers `mixed moments', i.e.\ integrals of $z \mapsto z^k\bar{z}^l$ with respect to $\mu_M$, for various $k,l \in \N$. Subject to a sufficient decay of the measures in question, this information is enough to uniquely recover the limiting measure. Due to the normality of $M$, the mixed moments can be expressed in terms of traces of products of the form $M^k{M^{\dagger}}^l$, so one can proceed in the usual manner. In this way, one can bypass the need to prove any kind of uniform integrability statement. However, for the purpose of proving Theorem \ref{thm:asymp_det}, we require uniform integrability of the logarithm near zero anyway, and the case of general shifts follows with little extra effort. Hence, in order to make our argument more streamlined, we pursue the Hermitisation approach.

\smallskip

\noindent\textbf{Organisation.} The paper is organised as follows. In Section \ref{sec:not_and_prelim}, we introduce some further notation and definitions/facts concerning probabilistic weak convergence and discrete Fourier analysis. In Section \ref{sec:unif_int_log}, we establish uniform integrability of the logarithm with respect to shifted singular value measures of the sequence $(C_n)_{n\geq1}$. In Section \ref{sec:conv_in_exp}, we prove Theorem \ref{thm:conv_in_exp}. Section \ref{sec:conv_in_prob} is devoted to the proof of Theorem \ref{thm:conv_in_prob}. In Section \ref{sec:det}, we give the short deduction of Theorem \ref{thm:asymp_det} and comment on several aspects of this result. In Section \ref{sec:conc_rem}, we make some concluding remarks and indicate possible directions for future investigation. Finally, Appendix \ref{app:tech_lemmas} contains some standard technical results in a form suitable for our applications.

\section{Notation and preliminaries}\label{sec:not_and_prelim}

We use Vinogradov asymptotic notation. In particular, the expressions $A = O(B)$ and $B = \Omega(A)$ have their usual meanings and are used interchangeably with $A \ll B$ and $B \gg A$ respectively. If $\theta$ is an element of $\R$ or $\R/\Z$, we write $e(\theta) \vcentcolon= e^{2\pi i\theta}$. We will occasionally use $\P_{x_1,\ldots,x_k\in S}$ and $\E_{x_1,\ldots,x_k\in S}$ to denote probability/expectation with respect to an independent uniformly random choice of elements $x_1,\ldots,x_k$ from a finite set $S$.

By default, a topological space $X$ will be equipped with its Borel $\sigma$-algebra $\mathcal{B}(X)$. In particular, without exception, by `(probability) measure on $X$' we mean `Borel (probability) measure on $X$'. As usual, we denote by $C_b(X)$ the space of bounded continuous (real-valued) functions on $X$ equipped with the uniform norm.\footnote{If $X$ is compact (e.g.\ $\N^*$), this is the same as $C(X)$, the space of all continuous functions on $X$.} We abuse notation by identifying measures on $\N^*$ with non-negative functions on $\N^*$ in the obvious way. Given a subset $S \subseteq \C$ and a parameter $r \geq 0$, we write
\[B(S,r) \vcentcolon= \{z \in \C \mid \exists w \in S\ |z-w|<r\}\]
for the $r$-neighbourhood of $S$. In particular, if $S = \{w\}$ is a singleton, this is the open ball of radius $r$ centred at $w$, and we write $B(w,r)$ instead. Similarly, the closed ball with centre $w \in \C$ and radius $r \geq 0$ is written as
\[\overline{B}(w,r) \vcentcolon= \{z \in \C \mid |z-w| \leq r\},\]
and we define the punctured closed ball $\overline{B}^*(w,r) \vcentcolon= \overline{B}(w,r)\setminus\{w\}$. We use $\lambda$ to denote Lebesgue measure on $\C$; this will cause no confusion with eigenvalues.

\subsection{Probabilistic weak convergence}\label{subsec:prob_weak_conv}

In this section, we review some basic notions and facts concerning convergence of random probability measures. The survey \cite[\S2]{fleermann-kirsch} serves very well as a gentle introduction to this topic (see also \cite[\S1.1]{tao-vu-circular} for a brief overview). We will only be concerned with convergence towards deterministic probability measures. Thus, if not stated explicitly, limits are always understood to be deterministic.

Let $K$ be a separable metric space, which for us will be $\C$, $\R$ or $[0,\infty)$. We endow the set $\mathcal{M}_1(K)$ of probability measures on $K$ with the topology of weak convergence. It is a fact that, in this way, $\mathcal{M}_1(K)$ becomes a separable metrisable space. By a \emph{random probability measure on $K$} we mean a random variable with values in $\mathcal{M}_1(K)$. For us, the most important examples of random probability measures come from ESDs of random matrices. Indeed, if $M$ is a random real/complex matrix, one can check that $\mu_M$, $\nu_M$ are random probability measures on $\C$ and $[0,\infty)$ respectively. 

Given a random probability measure $\kappa$ on $K$, it is easy to see that the map
\[\mathcal{B}(K) \to [0,1], \quad B \mapsto \E\bigl[\kappa(B)\bigr]\]
is a probability measure on $K$, which we call the \emph{expectation} of $\kappa$ and accordingly denote by $\E[\kappa]$. Let $(\mu_n)_{n\geq1}$ be a sequence of random probability measures on $K$ and let $\mu$ be a deterministic probability measure on $K$. We say that $(\mu_n)_{n\geq1}$ \emph{converges weakly in expectation} to $\mu$ if $(\E[\mu_n])_{n\geq1}$ converges weakly to $\mu$. We say that $(\mu_n)_{n\geq1}$ converges \emph{weakly in probability} to $\mu$ if for all $\varepsilon > 0$ we have
\[\P(\pi(\mu_n,\mu) > \varepsilon) \to 0 \text{ as } n \to \infty,\]
where $\pi$ is any metric inducing the topology on $\mathcal{M}_1(K)$; the particular choice of metric is not important. Unwinding the definitions, it follows that $\mu_n \to \mu$ weakly in expectation if and only if
\[\E\Bigl[\int_Kf\,d\mu_n\Bigr] \to \int_Kf\,d\mu \text{ as } n \to \infty\]
for all $f \in C_b(K)$ and similarly that $\mu_n \to \mu$ weakly in probability if and only if
\[\int_Kf\,d\mu_n \to \int_Kf\,d\mu\]
in probability for all $f \in C_b(K)$. We will freely make use of the fact that weak convergence in probability implies weak convergence in expectation.

We will require the notion of uniform integrability, both in the deterministic and probabilistic sense. Recall that a measurable function $f \colon K \to \overline{\R}$ is said to be \emph{uniformly integrable in probability} with respect to a sequence $(\mu_n)_{n\geq1}$ of random probability measures if for all $\varepsilon > 0$ we have
\[\lim_{t\to\infty}\sup_{n\geq1}\P\Bigl(\int_{\{|f|>t\}}|f|\,d\mu_n > \varepsilon\Bigr) = 0.\] 
If $(\mu_n)_{n\geq1}$ are in fact deterministic, we simply say $f$ is \emph{uniformly integrable} with respect to  $(\mu_n)_{n\geq1}$; this amounts to
\[\lim_{t\to\infty}\sup_{n\geq1}\int_{\{|f|>t\}}|f|\,d\mu_n = 0.\]
These definitions have as a consequence the following fact (see \cite[Remark 4.2]{bordenave-chafai} as well as the discussion on \cite[p.\ 20]{bordenave-chafai} immediately preceding it).

\begin{lemma}
\label{lm:unif_int_prob}
Let $(\mu_n)_{n\geq1}$ be random probability measures and let $\mu$ be a deterministic probability measure on $K$. Let $f \colon K \to \overline{\R}$ be a continuous function. If $\mu_n \to \mu$ weakly in probability and $f$ is uniformly integrable in probability with respect to $(\mu_n)_{n\geq1}$, then $f$ is integrable with respect to $\mu$ and
\[\int_Kf\,d\mu_n \to \int_Kf\,d\mu\]
in probability.
\end{lemma}

Specialising to non-random probability measures, we readily obtain the following corollary.

\begin{corollary}
\label{cor:unif_int}
Let $(\mu_n)_{n\geq1}$ and $\mu$ be probability measures on $K$ and let $f \colon K \to \overline{\R}$ be a continuous function. If $\mu_n \to \mu$ weakly and $f$ is uniformly integrable with respect to $(\mu_n)_{n\geq1}$, then $f$ is integrable with respect to $\mu$ and
\[\int_Kf\,d\mu_n \to \int_Kf\,d\mu \text{ as } n \to \infty.\]
\end{corollary}

As mentioned in Section \ref{sec:intro}, our main tool to prove convergence of ESDs in probability will be the Hermitisation method. Specifically, we will use the following criterion, which appears as \cite[Proposition 13.1]{sah-sahasrabudhe-sawhney} and is based on \cite[Lemma 4.3]{bordenave-chafai}.\footnote{These results are stated for sequences in which the $n$-th element is an $n\times n$ matrix. We require a variant in which the dimension of the $n$-th matrix is not necessarily $n$ (the proof carries over verbatim to our setting).}

\begin{proposition}
\label{prop:herm}
For each $n \geq 1$ let $M_n$ be a random complex matrix. Suppose that for $\lambda$-a.e.\ $z \in \C$, the following holds:
\begin{enumerate}[(i)]
    \item there exists a probability measure $\nu_z$ on $[0,\infty)$ such that $\nu_{M_n-zI} \to \nu_z$ weakly in probability;
    \item $\log$ is uniformly integrable in probability with respect to $(\nu_{M_n-zI})_{n\geq1}$.
\end{enumerate}
Then there exists a probability measure $\mu$ on $\C$ such that $\mu_{M_n} \to \mu$ weakly in probability.
\end{proposition}

\subsection{Fourier analysis on finite abelian groups}\label{subsec:fourier_analysis}

Let $G$ be a finite abelian group and let $\widehat{G}$ denote the group of characters of $G$.\footnote{These will be written using additive and multiplicative notation, respectively.} We briefly discuss some facts about the Fourier transform on $G$, as well as its relation to eigenvalues of $G$-circulant matrices. For this purpose, it is convenient to equip $G$ with the counting measure and $\widehat{G}$ with the uniform probability measure. Hence, the Fourier transform of a function $f \colon G \to \C$ is
\[\widehat{f} \colon \widehat{G} \to \C, \quad \gamma \mapsto \sum_{\gamma\in G}f(x)\overline{\gamma(x)}\]
and one can recover $f$ from $\widehat{f}$ via the inversion formula
\[f(x) = \E_{\gamma\in\widehat{G}}\widehat{f}(\gamma)\gamma(x).\]
Given another function $g \colon G \to \C$, its convolution with $f$ is defined to be
\[f*g \colon G \to \C, \quad x \mapsto \sum_{y\in G}f(x-y)g(y).\]
The Fourier transform converts convolution into pointwise multiplication: $\reallywidehat{f*g} = \widehat{f}\widehat{g}$. This means that the eigenvalues of the linear operator
\[T_f \colon \C^G \to \C^G, \quad g \mapsto f*g\]
are precisely the Fourier coefficients $\widehat{f}(\gamma)$ for $\gamma \in \widehat{G}$. Note that the adjoint of $T_f$ is $T_{\widebar{\widetilde{f}}}$ with $\widetilde{f}$ defined in accordance with \eqref{eq:tilde}. Consequently, $T_f$ is normal, so its singular values are the absolute values of its eigenvalues. Hence, for a $G$-circulant matrix $A$ corresponding to $T_a$, we have the equalities of multisets
\[\{\lambda_k(A) \mid 1 \leq k \leq |G|\} = \{\widehat{a}(\gamma) \mid \gamma \in \widehat{G}\}, \quad\quad \{\sigma_k(A) \mid 1 \leq k \leq |G|\} = \{|\widehat{a}(\gamma)| \mid \gamma \in \widehat{G}\}.\]
In particular, the eigenvalues of $C_n$ are $\widehat{S_n}(\gamma)$ for $\gamma \in \widehat{G_n}$. Recalling the definition \eqref{eq:iid_sum}, it follows that
\[\widehat{S_n}(\gamma) = \sum_{j=1}^{d}\widehat{1_{X_{n,j}}}(\gamma) = \sum_{j=1}^{d}\overline{\gamma(X_{n,j})}\] 
has distribution $\eta_{\mathrm{ord}_{\widehat{G_n}}(\gamma)}^{*d}$ and takes values in the set $dR_{\mathrm{ord}_{\widehat{G_n}}(\gamma)} \subseteq dR_{\exp(G_n)}$. Finally, since $\widehat{G}$ is isomorphic to $G$, we have $\rho_{\widehat{G}} = \rho_G$, a fact we will often use without mention.

\section{Uniform integrability of the logarithm}\label{sec:unif_int_log}

The goal of this section is to establish the uniform integrability of the logarithm with respect to shifted singular value measures of the sequence $(C_n)_{n\geq1}$. This provides the necessary input (ii) to Proposition \ref{prop:herm} and also to Theorem \ref{thm:asymp_det}. In turn, it naturally leads to the analysis of the distribution of sums of $m$-th roots of unity, where the number of roots is fixed and $m$ tends to infinity. Despite receiving a fair amount of attention (see \cite{tao-roots} and the references therein), this topic remains somewhat mysterious. In particular, even though standard heuristics suggest that the smallest such non-zero sum should be polynomially small in $m$, the best known lower bound is only exponentially small. Nevertheless, we will be able to get away with using this fairly crude estimate along with some elementary properties of the distribution of these sums.

We start by proving a couple of simple small ball probability estimates. The first one is effective in the regime when the radius decays roughly as $m^{-1}$, whereas the second one gives better results when $r$ is on the order of $m^{-2}$ or smaller. They are obtained by conditioning on the values of all but one or two terms in the sum and using that the sums of the remaining roots of unity are sufficiently spread out. Specifically, for the second estimate, we exploit the fact that $R_m$ is a Sidon set.

\begin{lemma}
\label{lm:small_ball_prob}
Let $m \in \N$, $z \in \C$ and $r > 0$. Then
\begin{equation}\label{eq:small_ball_lin}
    \eta_m^{*d}\bigl(\overline{B}(z,r)\bigr) \leq r + \frac{1}{m}.
\end{equation}
Furthermore, if $d \geq 2$, then
\begin{equation}\label{eq:small_ball_quad}
    \eta_m^{*d}\bigl(\overline{B}(z,r)\bigr) \ll \Bigl(rm+\frac{1}{m}\Bigr)^2.
\end{equation}
\end{lemma}
\begin{proof}
    Observe that for any $d'<d$ and $w \in \C$ we have
    \[\eta_m^{*d}\bigl(\overline{B}(w,r)\bigr) = (\eta_m^{*d'} * \eta_m^{*d-d'})\bigl(\overline{B}(w,r)\bigr) = \int_{\C}\eta_m^{*d'}\bigl(\overline{B}(w-w',r)\bigr)\,d\eta_m^{*d-d'}(w').\]
    As a consequence, we have
    \[\sup_{w\in\C}\eta_m^{*d}\bigl(\overline{B}(w,r)\bigr) \leq \sup_{w\in\C}\eta_m^{*d'}\bigl(\overline{B}(w,r)\bigr).\]
    Therefore, to prove \eqref{eq:small_ball_lin} and \eqref{eq:small_ball_quad}, we may assume that $d = 1$ and $d = 2$ respectively. Hence, the left-hand side in \eqref{eq:small_ball_lin} equals $|E_1|/m$, where
    \[E_1 \vcentcolon= \{k \in [m] \mid |e(k/m)-z| \leq r\}.\]
    If $E_1 = \varnothing$, we are done. Otherwise, choose a reference element $k_0 \in E_1$ and note that, by the triangle inequality, each $k \in E_1$ satisfies
    \[|e((k-k_0)/m)-1| = |e(k/m)-e(k_0/m)| \leq |e(k/m)-z| + |e(k_0/m)-z| \leq 2r,\]
    whence $\sin(|k-k_0|\pi/m) \leq r$. Using the inequality $\sin x \geq 2x/\pi$, which holds for all $x \in [0,\pi/2]$, we obtain that $\min(|k-k_0|,m-|k-k_0|) \leq rm/2$. Therefore, we have
    \[|E_1| \leq 2\Bigl\lfloor\frac{rm}{2}\Bigr\rfloor+1 \leq rm+1,\]
    so \eqref{eq:small_ball_lin} follows on dividing through by $m$. 
    
    Turning to \eqref{eq:small_ball_quad}, it suffices to show that $|E_2| \ll (rm^2+1)^2$, where now
    \[E_2 \vcentcolon= \{(k_1,k_2) \in [m]^2 \mid |e(k_1/m)+e(k_2/m)-z| \leq r\}.\]
    To this end, consider the map
    \[\alpha \colon E_2 \to \overline{B}(z,r), \quad (k_1,k_2) \mapsto e(k_1/m) + e(k_2/m).\]
    Then the preimage of each point of $\overline{B}(z,r)$ under $\alpha$ has cardinality at most $2$. This can easily be seen geometrically: any given point is the midpoint of at most one (possibly degenerate) chord of the unit circle. Moreover, by \cite[Proposition 3]{barber}, the image of $\alpha$ is $\Omega(m^{-2})$-separated. 
    
    The task can now be accomplished by a standard volume packing argument. For any distinct $w,w' \in \alpha(E_2)$, we have $B(w,cm^{-2}) \cap B(w',cm^{-2}) = \varnothing$, where $c > 0$ is an absolute constant. Since $B(w,cm^{-2}) \subseteq B(z,r+cm^{-2})$ for any $w \in \alpha(E_2)$, it follows that
    \begin{align*}
        |\alpha(E_2)|(cm^{-2})^2\pi &= \lambda\Biggl(\bigcup_{w\in\alpha(E_2)}B(w,cm^{-2})\Biggr)
        \leq \lambda(B(z,r+cm^{-2})) = (r+cm^{-2})^2\pi.
    \end{align*}
    Hence, we have
    \[|E_2| \leq 2|\alpha(E_2)| \leq \frac{2(r+cm^{-2})^2}{(cm^{-2})^2} \ll (rm^2+1)^2,\]
    as desired.
\end{proof}

The following observation about small sums of $m$-th roots of unity appears in \cite{myerson} and has since been reiterated several times, for example in \cite{barber} (see also the discussion in \cite{tao-roots}). Since the argument is very short, we include it for the convenience of the reader.

\begin{lemma}
\label{lm:sum_roots_unity}
If $z \in (dR_m)\setminus\{0\}$, then $|z| \geq d^{1-\varphi(m)}$, where $\varphi$ is Euler's totient function.
\end{lemma}
\begin{proof}
    Observe that $z$ is a member of the cyclotomic field $\Q(e(1/m))$. Each of its $\varphi(m)$ Galois conjugates is also a member of $dR_m$, so has absolute value at most $d$. Since $z$ is a non-zero algebraic integer, its norm is a non-zero integer. Hence, the product of all its conjugates (including $z$ itself) must have absolute value at least $1$, so the desired conclusion follows.
\end{proof}

Finally, we note that for a random variable $Y$ taking values in $[0, \infty)$ and a real number $T \geq 1$, layer cake representation gives the useful identity
\begin{equation}\label{eq:int_of_log}
    \E[(\log Y)1_{\{Y \geq T\}}] = \P(Y \geq T)\log T + \int_{T}^{\infty}\frac{\P(Y \geq u)}{u}\,du.
\end{equation}
We now have all the ingredients required to prove the main result of this section.

\begin{proposition}
\label{prop:unif_int}
Call $z \in \C$ \emph{good} if the function $\log$ is uniformly integrable in probability with respect to $(\nu_{C_n-zI})_{n\geq1}$. Then $\lambda$-a.e.\ $z \in \C$ is good. Furthermore, if $0 \not\in \bigcup_{n=1}^{\infty}dR_{\exp(G_n)}$, then $0$ is good.
\end{proposition}
\begin{proof}
    By Markov's inequality, $z$ is certainly good if
    \[\lim_{t\to\infty}\sup_{n\geq1}\E\Bigl[\int_{\{|\log|>t\}}|\log|\,d\nu_{C_n-zI}\Bigr] = 0,\]
    that is to say
    \[\lim_{t\to\infty}\sup_{n\geq1}\frac{1}{|G_n|}\sum_{k=1}^{|G_n|}\E\bigl[|\log\sigma_k(C_n-zI)|1_{\{|\log\sigma_k(C_n-zI)|\geq t\}}\bigr] = 0.\]
    This holds if and only if
    \[\lim_{T\to\infty}\sup_{n\geq1}\frac{1}{|G_n|}\sum_{\gamma\in\widehat{G_n}}\E\bigl[\log(|\widehat{S_n}(\gamma)-z|)1_{\{|\widehat{S_n}(\gamma)-z|\geq T\}}\bigr] = 0\]
    and
    \[\lim_{T\to\infty}\sup_{n\geq1}\frac{1}{|G_n|}\sum_{\gamma\in\widehat{G_n}}\E\Biggl[\log\Biggl(\frac{1}{|\widehat{S_n}(\gamma)-z|}\Biggr)1_{\{|\widehat{S_n}(\gamma)-z|\leq 1/T\}}\Biggr] = 0.\]
    The former is satisfied regardless of the choice of $z$ because the summand is equal to zero as soon as $T > d+|z|$. To deal with the latter, let $I(\gamma,T)$ denote the summand in the above sum. Consider the sets
    \[E_1 \vcentcolon= \bigcup_{n=1}^{\infty}dR_{\exp(G_n)}, \quad E_2 \vcentcolon= \limsup_{m\to\infty}B(dR_m,m^{-3d}).\]
    It suffices to show that if $z \in \C\setminus E_1$ is such that either $z \not\in E_2$ or $z = 0$, then $z$ is good. Indeed, the set $E_1$ is countable, so $\lambda(E_1) = 0$. In a similar vein, the union bound implies that
    \[\lambda(B(dR_m,m^{-3d})) \leq \sum_{w\in dR_m}\lambda(B(w,m^{-3d})) = |dR_m| \cdot (m^{-3d})^2\pi \ll m^d \cdot m^{-6d} = m^{-5d}.\]
    This means that 
    \[\sum_{m=1}^{\infty}\lambda(B(dR_m,m^{-3d})) \ll \sum_{m=1}^{\infty}m^{-5d} < \infty,\] 
    whence the first Borel--Cantelli lemma implies that $\lambda(E_2) = 0$ as well.
    
    Hence, fix a complex number $z \not\in E_1$ such that either $z \not\in E_2$ or $z = 0$. If $z \not \in E_2$, there exists $M_1 \in \N$ such that $|\widehat{S_n}(\gamma)-z| \geq \mathrm{ord}_{\widehat{G_n}}(\gamma)^{-3d}$ whenever $\mathrm{ord}_{\widehat{G_n}}(\gamma) \geq M_1$. If $z = 0$, then by Lemma \ref{lm:sum_roots_unity}, we have $|\widehat{S_n}(\gamma)-z| \geq (3d)^{-\mathrm{ord}_{\widehat{G_n}}(\gamma)}$. Given $\varepsilon > 0$, choose $T_1 \geq 1$ such that 
    \[\frac{\log T_1 + 1}{T_1} < \frac{\varepsilon}{3}\]
    and $M_2 \in \N$ large enough so that $M_2 \geq \max(T_1^{1/(3d)}, 3d)$ and
    \[\frac{\log m}{m} < \frac{\varepsilon}{9d}, \quad\quad \frac{1}{m} < \frac{\varepsilon}{12C\log(3d)}\]
    whenever $m \geq M_2$, where $C > 0$ is the implied constant in \eqref{eq:small_ball_quad}. Set $M \vcentcolon= \max(M_1,M_2)$ and for each $n \geq 1$ let
    \[L_n \vcentcolon= \{\gamma \in \widehat{G_n} \mid \mathrm{ord}_{\widehat{G_n}}(\gamma) \geq M\}.\]
    Suppose that $\gamma \in L_n$. The identity \eqref{eq:int_of_log} furnishes the expression
    \[I(\gamma,T_1) = \P(|\widehat{S_n}(\gamma)-z|\leq1/T_1)\log T_1 + \int_{T_1}^{\infty}\frac{\P(|\widehat{S_n}(\gamma)-z|\leq1/u)}{u}\,du.\]
    We decompose the integral as
    \[\int_{T_1}^{\mathrm{ord}_{\widehat{G_n}}(\gamma)^{3d}}\frac{\P(|\widehat{S_n}(\gamma)-z|\leq1/u)}{u}\,du + \int_{\mathrm{ord}_{\widehat{G_n}}(\gamma)^{3d}}^{(3d)^{\mathrm{ord}_{\widehat{G_n}}(\gamma)}}\frac{\P(|\widehat{S_n}(\gamma)-z|\leq1/u)}{u}\,du.\]
    By the first part of Lemma \ref{lm:small_ball_prob}, the contribution of the first two summands in the ensuing expression for $I(\gamma, T_1)$ is at most
    \[\Bigl(\frac{1}{T_1}+\frac{1}{\mathrm{ord}_{\widehat{G_n}}(\gamma)}\Bigr)\log T_1 + \int_{T_1}^{\mathrm{ord}_{\widehat{G_n}}(\gamma)^{3d}}\frac{1}{u}\Bigl(\frac{1}{u}+\frac{1}{\mathrm{ord}_{\widehat{G_n}}(\gamma)}\Bigr)\,du \leq \frac{\log T_1+1}{T_1}+\frac{3d\log\mathrm{ord}_{\widehat{G_n}}(\gamma)}{\mathrm{ord}_{\widehat{G_n}}(\gamma)}<\frac{2\varepsilon}{3}.\]
    The third summand is at most
    \[\P(|\widehat{S_n}(\gamma)-z| < \mathrm{ord}_{\widehat{G_n}}(\gamma)^{-3d})\cdot\mathrm{ord}_{\widehat{G_n}}(\gamma)\log(3d).\]
    If either $z \not\in E_2$ or $(z,d) = (0,1)$, then this equals zero. In the complementary case, when $z = 0$ and $d \geq 2$, the second part of Lemma \ref{lm:small_ball_prob} gives the upper bound
    \[\frac{4C}{\mathrm{ord}_{\widehat{G_n}}(\gamma)^2}\cdot\mathrm{ord}_{\widehat{G_n}}(\gamma)\log(3d) < \frac{\varepsilon}{3}.\]
    By combining these two estimates, we obtain that $I(\gamma,T_1)<\varepsilon$ whenever $\gamma \in L_n$. In particular,
    \[\frac{1}{|G_n|}\sum_{\gamma\in L_n}I(\gamma,T) < \varepsilon\]
    holds for $T = T_1$. Since $I(\gamma,\cdot)$ is decreasing, this is also true for any $T \geq T_1$. Finally, since $z \not\in E_1$, we may choose $T_2 > 0$ so that 
    \[\frac{1}{T_2} < \min\{|z-w| \mid  n \geq 1,\ w \in dR_{\mathrm{ord}_{\widehat{G_n}}(\gamma)},\ \gamma\in\widehat{G_n}\setminus L_n\}.\]
    It follows that for $T \geq T_2$ and $\gamma \in \widehat{G_n}\setminus L_n$ we have $I(\gamma, T) = 0$. Hence, for $T \geq \max(T_1,T_2)$, we obtain that
    \[\frac{1}{|G_n|}\sum_{\gamma\in\widehat{G_n}}I(\gamma,T) < \varepsilon,\]
    thereby completing the proof.
\end{proof}

\section{Convergence in expectation}\label{sec:conv_in_exp}

In this section, we prove Theorem \ref{thm:conv_in_exp}. The starting point is the observation that, for $f \in C_b(\C)$, we may express
\[\E\Bigl[\int_{\C}f\,d\mu_{C_n}\Bigr] = \frac{1}{|G_n|}\sum_{k=1}^{|G_n|}\E\bigl[f(\lambda_k(C_n))\bigr] = \sum_{m\in\N^*}\rho_{G_n}(m)\int_{\C}f\,d\eta_m^{*d} = \int_{\N^*}\Phi(f)\,d\rho_{G_n}.\]
Here, we define the linear operator
\[\Phi \colon \mathcal{L}^{\infty}(\C) \to \R^{\N^*}, \quad f \mapsto \Bigl(\N^* \to \R,\ m \mapsto \int_{\C}f\,d\eta_m^{*d}\Bigr),\]
where $\mathcal{L}^{\infty}(\C)$ denotes the space of all bounded measurable functions $\C \to \R$. Thus, if $\mu$ is a probability measure on $\C$, we have that $\mu_{C_n} \to \mu$ weakly in expectation if and only if
\[\int_{\N^*}\Phi(f)\,d\rho_{G_n} \to \int_{\C}f\,d\mu\]
for all $f \in C_b(\C)$. One direction of Theorem \ref{thm:conv_in_exp} is now clear. Indeed, note that $\eta_m^{*d} \to \eta_{\infty}^{*d}$ weakly as $m \to \infty$, so $\Phi$ restricts to a linear operator $C_b(\C) \to C(\N^*)$. Hence, if $\rho_{G_n} \to \rho$ weakly, then for all $f \in C_b(\C)$ we have
\[\int_{\N^*}\Phi(f)\,d\rho_{G_n} \to \int_{\N^*}\Phi(f)\,d\rho = \sum_{m\in\N^*}\rho(m)\int_{\C}f\,d\eta_m^{*d} = \int_{\C}f\,d\mu,\]
where $\mu$ is as in \eqref{eq:limit_in_exp}. Therefore, $\mu_{C_n} \to \mu$ weakly in expectation, as desired.

To establish the converse, we will show that if $(\int_{\N^*}g\,d\rho_{G_n})_{n\geq1}$ converges for any $g \in \mathrm{im}(\Phi|_{C_b(\C)})$, then $(\rho_{G_n}(m))_{n\geq1}$ converges for any $m \in \N$. The pointwise limit will then be a sub-probability measure on $\N$, so one will be able to uniquely extend it to a probability measure on $\N^*$, towards which $(\rho_{G_n})_{n\geq1}$ will converge weakly. But observe that the functions $g \in C(\N^*)$ for which $(\int_{\N^*}g\,d\rho_{G_n})_{n\geq1}$ converges form a closed linear subspace of $C(\N^*)$. It thus suffices to show that, for each $m \in \N$, the function $\delta_m$ lies in closure of $\mathrm{im}(\Phi|_{C_b(\C)})$ inside $C(\N^*)$. This can be achieved by a two-step approximation argument. First, we identify the image $g_k$ of $1_{\{de(1/k)\}}$ under $\Phi$, where $k \in \N$ is arbitrary. By approximating $1_{\{de(1/k)\}}$ by suitable continuous functions, we show that $g_k$ lies in the closure of $\mathrm{im}(\Phi|_{C_b(\C)})$ inside $C(\N^*)$. The second step then consists of showing that $\delta_m$ lies in the closed linear span of $g_k$ as $k$ ranges over the positive integers.

Motivated by the above discussion, for each positive integer $m$ we define the function
\[g_m \colon \N^* \to \R, \quad n \mapsto \begin{cases}\frac{1}{n^d} & \text{if } n \in \N \text{ and } m \mid n\\0 & \text{else}\end{cases}\]
and remark that it belongs to $C(\N^*)$. Moreover, we claim that $\Phi(1_{\{de(1/m)\}}) = g_m$. Indeed, both functions clearly vanish at $\infty$, while for $n \in \N$, one has
\[\Phi(1_{\{de(1/m)\}})(n) = \eta_n^{*d}(\{e(1/m)\}) = \frac{1}{n^d}\sum_{k_1,\ldots,k_d\in[n]}\delta_{e(k_1/n)+\ldots+e(k_d/n)}(\{de(1/m)\}).\]
By (the equality case in) the triangle inequality, the summand equals $1$ if $k_1/n = \ldots = k_d/n = 1/m$ and $0$ otherwise. Hence, the claim follows. 

We are ready to carry out the first step of the argument.

\begin{lemma}
\label{lm:tent_fn}
For each $m \in \N$, the function $g_m$ lies in the closure of $\mathrm{im}(\Phi|_{C_b(\C)})$ in $C(\N^*)$.
\end{lemma}
\begin{proof}
    Fix $m$ and consider for each $k \in \N$ the tent function
    \[\psi_{m,k} \colon \C \to \R, \quad z \mapsto \max\bigl(1-k|z-de(1/m)|,0\bigr),\]
    which is a member of $C_b(\C)$. It suffices to show that
    \[\lim_{k\to\infty}\lVert \Phi(\psi_{m,k}-1_{\{de(1/m)\}})\rVert_{\infty} = 0.\]
    Since we have
    \[|\psi_{m,k}-1_{\{de(1/m)\}}| \leq 1_{\overline{B}^*(de(1/m),1/k)},\]
    it is enough to show that
    \[\lim_{k\to\infty}\sup_{n\in\N}\eta_n^{*d}\bigl(\overline{B}^*(de(1/m),1/k)\bigr) = 0.\]
    So let $\varepsilon > 0$ be given. By continuity of probability measures, for each fixed $n \in \N$ we have
    \[\lim_{k\to\infty}\eta_n^{*d}\bigl(\overline{B}^*(de(1/m),1/k)\bigr) = \eta_n^{*d}\Biggl(\bigcap_{k=1}^{\infty}\overline{B}^*(de(1/m),1/k)\Biggr) = \eta_n^{*d}(\varnothing) = 0.\]
    Hence, we may choose $k > 2/\varepsilon$ so that
    \[\eta_n^{*d}\bigl(\overline{B}^*(de(1/m),1/k)\bigr) < \varepsilon\]
    for all $n \leq 2/\varepsilon$. On the other hand, for $n > 2/\varepsilon$, Lemma \ref{lm:small_ball_prob} implies
    \[\eta_n^{*d}\bigl(\overline{B}^*(de(1/m),1/k)\bigr) \leq \frac{1}{k} + \frac{1}{n} < \varepsilon.\]
    In conclusion, we have
    \[\sup_{n\in\N}\eta_n^{*d}\bigl(\overline{B}^*(de(1/m),\varepsilon)\bigr) \leq \varepsilon,\]
    as desired.
\end{proof}

The second and final step is encoded by the following lemma.

\begin{lemma}
\label{lm:mobius_fn}
Let $\mu$ denote the Möbius function. Then for any $m, K \in \N$ we have
\[\Bigl\lVert m^d\sum_{k=1}^{K}\mu(k)g_{km} - \delta_m\Bigr\rVert_{\infty} \leq \frac{1}{K}.\]
In particular, for any $m \in \N$, the function $\delta_m$ lies in the closed linear span of $\{g_k \mid k \in \N\}$ in $C(\N^*)$.
\end{lemma}
\begin{proof}
    Fix $m$, $K$ and consider the function
    \[h = \delta_m - m^d\sum_{k=1}^{K}\mu(k)g_{km},\]
    which clearly vanishes at $\infty$. Given a positive integer $n$, one computes
    \[m^d\sum_{k=1}^{K}\mu(k)g_{km}(n) = \begin{cases}\bigl(\frac{m}{n}\bigr)^d\sum_{k\in[K],\ k\mid \frac{n}{m}}\mu(k) & \text{if } m \mid n\\0 & \text{else}\end{cases}.\]
    Thus, if $m$ doesn't divide $n$, then certainly $h(n) = 0$. On the other hand, if $m$ divides $n$, then by the well-known fact that
    \[\sum_{s\mid r}\mu(s) = \delta_1(r)\]
    for any positive integer $r$, one obtains
    \[h(n) = \delta_m(n) - \Bigl(\frac{m}{n}\Bigr)^d\Biggl(\delta_m(n)-\sum_{k>K,\ k\mid \frac{n}{m}}\mu(k)\Biggr) = \Bigl(\frac{m}{n}\Bigr)^d\sum_{k>K,\ k\mid \frac{n}{m}}\mu(k).\]
    Since $|\mu(k)| \leq 1$ for any $k$, the triangle inequality implies that
    \[|h(n)| \leq \Bigl(\frac{m}{n}\Bigr)^d\Bigl|\Bigl\{k > K\ \Big|\ k \mid \frac{n}{m}\Bigr\}\Bigr|.\]
    Since the map $d \mapsto \frac{r}{d}$ is an involution on the set of divisors of $r \in \N$, it follows that
    \[|h(n)| \leq \Bigl(\frac{m}{n}\Bigr)^d\Bigl|\Bigl\{k < \frac{n}{Km}\ \Big|\ k \mid \frac{n}{m}\Bigr\}\Bigr| \leq \frac{m}{n} \cdot \frac{n}{Km} = \frac{1}{K},\]
    as desired.
\end{proof}

\section{Convergence in probability}\label{sec:conv_in_prob}

In this section, we prove Theorem \ref{thm:conv_in_prob}. The main step is to establish convergence in probability of shifted singular value measures of the sequence $(C_n)_{n\geq1}$, which then serves as the input (i) to Proposition \ref{prop:herm}. We begin by deriving an expression for the moments of shifted singular value measures. This is accomplished by a calculation of a standard type. We first fix some global notation. Let $n,k$ be positive integers. For each $z \in \C$ we define the random moment
\[W_{n,k,z} \vcentcolon= \int_{[0,\infty)}t^{2k}\,d\nu_{C_n-zI}(t).\]
Given a set $P \subseteq [2k]$ and a tuple $j \in [d]^P$, we introduce the event
\[E_{n,k,j} \vcentcolon= \Biggl\{\sum_{i=1}^{d}p(k,i,j)X_{n,i} = 0\Biggr\},\]
where for $i \in [d]$ we define
\[p(k,i,j) \vcentcolon= |\{l \in P \cap [k] \mid j_l = i\}| - |\{l \in P \setminus [k] \mid j_l = i\}|.\]
Finally, for the sake of brevity, we write
\[c(k,P,z) \vcentcolon= (-z)^{|\{1,\ldots,k\}\setminus P|}(-\overline{z})^{|\{k+1,\ldots,2k\}\setminus P|},\]
where we make the convention that $0^0 \vcentcolon= 1$.

\begin{lemma}
\label{lm:moments_of_esd}
We have the expression
\[W_{n,k,z} = \sum_{P\subseteq [2k]}c(k,P,z)\sum_{j\in[d]^P}1_{E_{n,k,j}}.\]
In particular, we have
\[\E\bigl[W_{n,k,z}\bigr] = \sum_{P\subseteq [2k]}c(k,P,z)\sum_{j\in[d]^P}\P(E_{n,k,j})\]
and
\[\mathrm{Var}(W_{n,k,z}) = \sum_{P,P'\subseteq [2k]}c(k,P,z)c(k,P',\overline{z})\sum_{j\in[d]^P,\ j'\in[d]^{P'}}\mathrm{Cov}(1_{E_{n,k,j}},1_{E_{n,k,j'}}).\]
\end{lemma}
\begin{proof}
    We have
    \begin{equation}\label{eq:additive_energy}
        W_{n,k,z} = \frac{1}{|G_n|}\mathrm{Tr}\Bigl[\bigl((C_n-zI)^{\dagger}(C_n-zI)\bigr)^k\Bigr] = \sum_{x\in\cX_{n,k}}\prod_{l=1}^{k}(S_n-z1_{\{0\}})(x_l)\prod_{l=k+1}^{2k}(S_n-\overline{z}1_{\{0\}})(x_l),
    \end{equation}
    where we define the set of additive $2k$-tuples in $G_n$ to be
    \[\cX_{n,k} \vcentcolon= \{(x_1,\ldots,x_k,x_{k+1},\ldots,x_{2k}) \in G_n^{2k} \mid x_1+\ldots+x_k = x_{k+1} + \ldots + x_{2k}\}.\]
    By expanding out the products, we obtain that the summand in \eqref{eq:additive_energy} equals
    \[\sum_{P\subseteq[2k]}c(k,P,z)\sum_{j\in[d]^P}\prod_{l\in P}1_{\{X_{n,j_l} = x_l\}}\prod_{l\in[2k]\setminus P}1_{\{0\}}(x_l).\]
    The desired conclusion now follows upon observing that
    \[\sum_{x\in\cX_{n,k}}\prod_{l\in P}1_{\{X_{n,j_l} = x_l\}}\prod_{l\in[2k]\setminus P}1_{\{0\}}(x_l) = 1_{E_{n,k,j}}\]
    and swapping the order of summation.
\end{proof}

\begin{remark}
\label{rem:additive_energy}
In a Fourier-analytic guise, \eqref{eq:additive_energy} can be viewed as a consequence of Parseval's identity and the fact that the Fourier transform diagonalises convolution. On the event \eqref{eq:diff_elements}, $W_{n,k,0}$ is precisely the $k$-fold additive energy of the subset of $G_n$ whose indicator function is $S_n$. Broadly similar considerations appear for example in \cite[\S4.5]{green-harmonic-analysis}.
\end{remark}

We now arrive at a criterion for weak convergence of $(\mu_{C_n})_{n\geq1}$ in probability, which is the main stepping stone towards Theorem \ref{thm:conv_in_prob}. The proof of the more difficult backward implication entails the Hermitisation strategy and the subsidiary method of moments.

\begin{proposition}
\label{prop:crit_conv_prob}
The following are equivalent:
\begin{enumerate}[(i)]
    \item $(\mu_{C_n})_{n\geq1}$ converges weakly in probability;
    \item $(\mu_{C_n})_{n\geq1}$ converges weakly in expectation and for all $k \in \N$, $z \in \C$ we have
    \begin{equation}\label{eq:var_conv2}
        \lim_{n\to\infty}\mathrm{Var}(W_{n,k,z}) = 0.
    \end{equation}
\end{enumerate}
\end{proposition}
\begin{proof}
    We first prove that (i) implies (ii), so assume there exists a probability measure $\mu$ on $\C$ such that $\mu_n \to \mu$ weakly probability. Since weak convergence in probability implies weak convergence in expectation, it follows that $\mu_{C_n} \to \mu$ weakly in expectation. Furthermore, note that we have
    \begin{equation}\label{eq:shifted_moment}
        W_{n,k,z} = \int_{\C}|w-z|^{2k}\,d\mu_{C_n}(w).
    \end{equation}
    Since $\mu_{C_n}$ is supported on $\overline{B}(0,d)$, the integrand is uniformly integrable in probability with respect to $(\mu_{C_n})_{n\geq1}$. By Lemma \ref{lm:unif_int_prob}, it follows that
    \[W_{n,k,z} \to \int_{\C}|w-z|^{2k}\,d\mu(w)\]
    in probability, where the integral on the right-hand side is finite. Since the random variables $(W_{n,k,z})_{n\geq1}$ are uniformly bounded (say by $(d+|z|)^{2k}$), it follows that \eqref{eq:var_conv2} holds, as desired.

    Conversely, suppose that (ii) holds. By Propositions \ref{prop:herm} and \ref{prop:unif_int}, it is enough to show that for all $z \in \C$ there is a probability measure $\nu_z$ on $[0,\infty)$ such that $\nu_{C_n-zI} \to \nu_z$ weakly in probability. To this end, we use the method of moments, in the form of Proposition \ref{prop:moment_method}. It suffices to show that for each $k \geq 1$ the limit
    \[b_k \vcentcolon= \lim_{n\to\infty}\E\bigl[W_{n,k,z}\bigr]\]
    exists and that $(b_k)_{k\geq1}$ satisfies the condition (ii) of Proposition \ref{prop:moment_method}. Note that the expression \eqref{eq:shifted_moment} implies that
    \[\E\bigl[W_{n,k,z}\bigr] = \int_{\C}|w-z|^{2k}\,d\E\bigl[\mu_{C_n}\bigr](w).\]
    By assumption, the sequence $(\E\bigl[\mu_{C_n}\bigr])_{n\geq1}$ converges weakly. Similarly as before, the integrand in the above identity is uniformly integrable with respect to this sequence, so Corollary \ref{cor:unif_int} implies the existence of the limit $b_k$. Finally, by Lemma \ref{lm:moments_of_esd}, we may crudely estimate
    \[b_k \leq \sum_{P\subseteq[2k]}|c(k,P,z)|\bigl|[d]^P\bigr| \leq \sum_{P\subseteq[2k]}\max(|z|,1)^{2k-|P|}d^{|P|} = (\max(|z|,1)+d)^{2k},\]
    so we in fact have $\limsup_{k\to\infty}b_{2k}^{1/(2k)}/(2k) = 0$. This concludes the proof.
\end{proof}

We have seen in Lemma \ref{lm:moments_of_esd} that the expectation/variance of $W_{n,k,z}$ can be expressed in terms of densities of solutions to certain linear equations over $G_n$, namely those appearing in the events $E_{n,k,j}$. The following lemma converts these densities into a more tractable form. As a consequence, it also gives a necessary and sufficient condition for the convergence of $\mathrm{Var}(W_{n,k,z})$ to zero. A crucial point for the necessity of this condition is that the events $E_{n,k,j}$ for fixed $n, k$ are pairwise positively correlated.

\begin{lemma}
\label{lm:covariance}
Let $k$ be a positive integer. Then for any $n \geq 1$, $P,P' \subseteq [2k]$ and $j \in [d]^P$, $j' \in [d]^{P'}$, the covariance $\mathrm{Cov}(1_{E_{n,k,j}},1_{E_{n,k,j'}})$ can be expressed as the difference between
\begin{equation}\label{eq:two_by_two_prob}
    \P_{x,x'\in G_n}(\forall i \in [d]\ p(k,i,j)x+p(k,i,j')x' = 0)
\end{equation}
and
\begin{equation}\label{eq:prod_prob}
    \P_{x\in G_n}(\forall i \in [d]\ p(k,i,j)x = 0)\P_{x'\in G_n}(\forall i \in [d]\ p(k,i,j')x' = 0).
\end{equation}
In particular, it is non-negative, i.e.\ the events $E_{n,k,j}$, $E_{n,k,j'}$ are positively correlated. Furthermore,
\begin{equation}\label{eq:var_conv1}
    \lim_{n\to\infty}\mathrm{Var}(W_{n,k,z}) = 0
\end{equation}
holds for all $z \in \C$ if and only if
\begin{equation}\label{eq:covar_conv}
    \lim_{n\to\infty}\mathrm{Cov}(1_{E_{n,k,j}},1_{E_{n,k,j'}}) = 0.
\end{equation}
holds for all $P,P' \subseteq [2k]$ and $j \in [d]^P$, $j' \in [d]^{P'}$.
\end{lemma}
\begin{proof}
    We have
    \[\mathrm{Cov}(1_{E_{n,k,j}},1_{E_{n,k,j'}}) = \P(E_{n,k,j} \cap E_{n,k,j'}) - \P(E_{n,k,j})\P(E_{n,k,j'}).\]
    The first statement now follows from Lemma \ref{lm:lin_sys} applied to the $1\times d$ matrices $(p(k,i,j))_{i\in[d]}$ and $(p(k,i,j'))_{i\in [d]}$ as well as the $2 \times d$ matrix comprising these matrices as rows. To prove the second statement, we use Lemma \ref{lm:moments_of_esd}, bearing in mind that the sums in the expression for $\mathrm{Var}(W_{n,k,z})$ are finite. It is clear that \eqref{eq:covar_conv} implies \eqref{eq:var_conv1}. The converse follows by specialising \eqref{eq:var_conv1} to $z = -1$ and using the fact that $\mathrm{Cov}(1_{E_{n,k,j}},1_{E_{n,k,j'}}) \geq 0$.
\end{proof}

To establish Theorem \ref{thm:conv_in_prob}, it remains to unfold what it means for the covariances of the events $E_{n,k,j}$ to converge to zero. We will argue that, in the presence of weak convergence in expectation of $(\mu_{C_n})_{n\geq1}$, this forces the limiting measure $\rho$ from Theorem \ref{thm:conv_in_exp} to satisfy a certain functional equation. From this information, it will be a simple matter to deduce that $\rho$ must be a Dirac mass.

To achieve this plan, it will be useful to introduce the following setup. Given $k \in \N$, $P,P' \subseteq [2k]$ and $j\in[d]^P$, $j'\in[d]^{P'}$, choose a $2 \times 2$ integer matrix $A^{(k,j,j')}$ whose rows generate the same subgroup of $\Z^2$ as $\{(p(k,i,j), p(k,i,j')) \mid i \in [d]\}$. By Smith normal form (see e.g.\ \cite[Theorem 3.8, Theorem 3.9]{jacobson}), unless $A^{(k,j,j')}$ is the zero matrix, there exist matrices $U^{(k,j,j')}, V^{(k,j,j')} \in \mathrm{GL}_2(\Z)$ such that $U^{(k,j,j')}A^{(k,j,j')}V^{(k,j,j')}$ is the diagonal matrix
\[\mathrm{diag}\Biggl(\gcd\bigl(A^{(k,j,j')}_{1,1}, A^{(k,j,j')}_{1,2}, A^{(k,j,j')}_{2,1}, A^{(k,j,j')}_{2,2}\bigr), \frac{\big|A^{(k,j,j')}_{1,1}A^{(k,j,j')}_{2,2}-A^{(k,j,j')}_{1,2}A^{(k,j,j')}_{2,1}\big|}{\gcd\bigl(A^{(k,j,j')}_{1,1}, A^{(k,j,j')}_{1,2}, A^{(k,j,j')}_{2,1}, A^{(k,j,j')}_{2,2}\bigr)}\Biggr).\]
Here, we use the convention that if $r \in \N$ and $a_1,\ldots,a_r \in \Z$, then $\gcd(a_1,\ldots,a_r)$ is defined to be the largest positive integer dividing each of $a_1,\ldots,a_r$, unless all $a_1,\ldots,a_r$ are zero, when it is defined to be zero. Since the action of $\mathrm{GL}_2(\Z)$ on $G_n^2$ preserves the uniform probability measure, the probability \eqref{eq:two_by_two_prob} can be expressed as
\[\P_{y\in G_n^2}\bigl(A^{(k,j,j')}y = 0\bigr) = \P_{y\in G_n^2}\bigl(U^{(k,j,j')}A^{(k,j,j')}V^{(k,j,j')}y = 0\bigr).\]
This, in turn, conveniently factors as
\[\tau_{G_n}\Bigl(\gcd\bigl(A^{(k,j,j')}_{1,1}, A^{(k,j,j')}_{1,2}, A^{(k,j,j')}_{2,1}, A^{(k,j,j')}_{2,2}\bigr)\Bigr)\tau_{G_n}\Biggl(\frac{\big|A^{(k,j,j')}_{1,1}A^{(k,j,j')}_{2,2}-A^{(k,j,j')}_{1,2}A^{(k,j,j')}_{2,1}\big|}{\gcd\bigl(A^{(k,j,j')}_{1,1}, A^{(k,j,j')}_{1,2}, A^{(k,j,j')}_{2,1}, A^{(k,j,j')}_{2,2}\bigr)}\Biggr).\]
Here, for a finite abelian group $G$ and an integer $r$, we write
\[\tau_G(r) \vcentcolon= \P_{x\in G}(rx = 0)\] 
for the proportion of $r$-torsion elements of $G$. Continuing in a similar fashion, since the subgroup of $\Z$ generated by $a_1,\ldots,a_r\in \Z$ is precisely $\gcd(a_1,\ldots,a_r)\Z$, the product \eqref{eq:prod_prob} equals
\[\tau_{G_n}\Bigl(\gcd\bigl(p(k,1,j), \ldots, p(k,d,j)\bigr)\Bigr)\tau_{G_n}\Bigl(\gcd\bigl(p(k,1,j'), \ldots, p(k,d,j')\bigr)\Bigr).\]
Since the projection of the subgroup of $\Z^2$ generated by a given set is equal to the subgroup of $\Z$ generated by its projection, we may rewrite this product as
\[\tau_{G_n}\Bigl(\gcd\bigl(A^{(k,j,j')}_{1,1}, A^{(k,j,j')}_{2,1}\bigr)\Bigr)\tau_{G_n}\Bigl(\gcd\bigl(A^{(k,j,j')}_{1,2}, A^{(k,j,j')}_{2,2}\bigr)\Bigr).\]

We are now in a position to prove Theorem \ref{thm:conv_in_prob}. By Theorem \ref{thm:conv_in_exp} and Proposition \ref{prop:crit_conv_prob}, it suffices to prove that if $(\rho_{G_n})_{n\geq1}$ converges weakly to a probability measure $\rho$ on $\N^*$, then $\rho$ is a Dirac mass if and only if for all $k \in \N$ and $z \in \C$ we have
\begin{equation}\label{eq:var_conv3}
    \lim_{n\to\infty}\mathrm{Var}(W_{n,k,z}) = 0.
\end{equation}
If we now define the function
\[\tau \colon \N_0 \to [0,1], \quad m \mapsto \begin{cases}1 & \text{if } m = 0\\\sum_{m'\in\N,\ m'\mid m}\rho(m') & \text{else}\end{cases},\]
then $(\tau_{G_n}|_{\N_0})_{n\geq1}$ converges pointwise to $\tau$. Moreover, by Lemma \ref{lm:covariance} and the above discussion, \eqref{eq:var_conv3} holds for all $k \in \N$ and $z \in \C$ if and only if
\begin{equation}\label{eq:lim_prod1}
    \tau\Bigl(\gcd\bigl(A^{(k,j,j')}_{1,1}, A^{(k,j,j')}_{1,2}, A^{(k,j,j')}_{2,1}, A^{(k,j,j')}_{2,2}\bigr)\Bigr)\tau\Biggl(\frac{\big|A^{(k,j,j')}_{1,1}A^{(k,j,j')}_{2,2}-A^{(k,j,j')}_{1,2}A^{(k,j,j')}_{2,1}\big|}{\gcd\bigl(A^{(k,j,j')}_{1,1}, A^{(k,j,j')}_{1,2}, A^{(k,j,j')}_{2,1}, A^{(k,j,j')}_{2,2}\bigr)}\Biggr)
\end{equation}
equals
\begin{equation}\label{eq:lim_prod2}
    \tau\Bigl(\gcd\bigl(A^{(k,j,j')}_{1,1}, A^{(k,j,j')}_{2,1}\bigr)\Bigr)\tau\Bigl(\gcd\bigl(A^{(k,j,j')}_{1,2}, A^{(k,j,j')}_{2,2}\bigr)\Bigr)
\end{equation}
for all $k \in \N$, $P,P' \subseteq [2k]$ and $j \in [d]^P$, $j' \in [d]^{P'}$ such that $A^{(k,j,j')} \neq 0$. Thus, our task boils down to showing that $\tau$ satisfies these relations if and only if it is of the form $1_{m\N_0}$ for some $m \in \N_0$. One direction is clear: if $\tau = 1_{m\N_0}$, then \eqref{eq:lim_prod1} equals $1$ if $m$ divides each entry of $A^{(k,j,j')}$ and $0$ otherwise, and the same is true of \eqref{eq:lim_prod2}. To show the other implication, we contend that $\tau$ must satisfy the functional equation
\begin{equation}\label{eq:func_eq}
    \tau(\gcd(a,b)) = \tau(a)\tau(b)
\end{equation}
for all $a,b \in \N$. Indeed, given any such $a,b$, we may specialise the equality of \eqref{eq:lim_prod1} and \eqref{eq:lim_prod2} to the case when $k = \max(a,b)$, $P = [a]$, $P' = [b]$ and $j = 1_P$, $j' = 1_{P'}$. Then for $i \in [d]$ we have
\[p(k,i,j) = \begin{cases}a & \text{if } i = 1\\0 & \text{else}\end{cases}, \quad p(k,i,j') = \begin{cases}b & \text{if } i = 1\\0 & \text{else}\end{cases}.\]
Hence, we may take $A^{(k,j,j')}$ to be the matrix $\begin{pmatrix}a & b\\0 & 0\end{pmatrix}$, and the desired equality \eqref{eq:func_eq} follows. Finally, to finish from \eqref{eq:func_eq}, note that we in particular obtain that whenever $a \mid b$, either $\tau(a) = 0$ or $\tau(b) = 1$. Taking $a = b$, it follows that $\tau$ takes values in the set $\{0,1\}$. Now if $\tau = 1_{\{0\}}$, we are done. Otherwise, take the least $m \in \N$ such that $\tau(m) = 1$. Then for any $a \in \N$ we get that $\tau(a) = \tau(\gcd(a,m))$. But since $\gcd(a,m) \leq m$, it follows by minimality of $m$ that $\tau(\gcd(a,m)) = 1$ if $\gcd(a,m) = m$ and $\tau(\gcd(a,m)) = 0$ otherwise. Thus, $\tau(a) = 1$ if and only if $m \mid a$, as desired.

\section{Determinants}\label{sec:det}

With Theorem \ref{thm:conv_in_prob} in place, it is now a short step to Theorem \ref{thm:asymp_det}. Observe that
\[\frac{1}{|G_n|}\log|\det(C_n)| = \frac{1}{|G_n|}\sum_{k=1}^{|G_n|}\log\sigma_k(C_n) = \int_{[0,\infty)}\log\,d\nu_{C_n}.\]
By Theorem \ref{thm:conv_in_prob}, we have that $\mu_{C_n} \to \eta_m^{*d}$ weakly in probability, so it follows that $\nu_{C_n} \to |\cdot|_*\eta_m^{*d}$ weakly in probability. By Proposition \ref{prop:unif_int}, we know that $\log$ is uniformly integrable in probability with respect to $(\nu_{C_n})_{n\geq1}$, so the desired conclusion follows from Lemma \ref{lm:unif_int_prob}.

Several remarks concerning the statement of Theorem \ref{thm:asymp_det} are in order. We start by elaborating on the somewhat cryptic condition that $dR_{\exp(G_n)}$ shouldn't contain $0$.

\begin{remark}
\label{rem:singularity}
The assumption that $0\not\in dR_{\exp(G_n)}$ in Theorem \ref{thm:asymp_det} means precisely that $C_n$ is never singular. Dropping this assumption can result in $C_n$ being singular with high probability.\footnote{In particular, conditioning on the non-singularity of $C_n$ does not necessarily yield a contiguous model.} For example, let $p$ be a prime and let $d$ be a multiple of $p$. For each $n \geq 1$ let $G_n = \mathbb{F}_p^n$. Then each non-trivial character $\gamma \in \widehat{G_n}$ has order $p$, so $\widehat{S_n}(\gamma)$ is zero with probability bounded away from zero. Since $\widehat{S_n}(\gamma_1), \widehat{S_n}(\gamma_2)$ are independent for distinct $\gamma_1,\gamma_2 \in \widehat{G_n}$, a straightforward second moment argument reveals that the probability of $C_n$ being non-singular tends to zero as $n \to \infty$.
\end{remark}

To simplify our discussion, we assume in the remainder of this section that $G_n = \Z/n\Z$. In particular, this means that $\exp(G_n) = n$ for all $n \geq 1$. We recall the Euler--Mascheroni constant $\gamma \approx 0.577$, which incidentally makes an appearance in the discussion of two unrelated aspects of Theorem \ref{thm:asymp_det}.

\begin{remark}
\label{rem:density}
Similarly as in Remark \ref{rem:singularity}, it is not hard to see that leaving out the condition $0 \not\in dR_n$ may lead to $C_n$ being singular with probability bounded away from zero. It is natural to wonder just how restrictive this condition is. In other words, for which $n \in \N$ is it possible to find $d$ many $n$-th roots of unity summing to zero? As per the main result of Lam and Leung \cite{lam-leung}, this can be done if and only if $d$ is a non-negative integer linear combination of the prime factors of $n$.

To estimate the density of $n$ for which this doesn't hold, note that this property depends only on the divisibility of $n$ by primes not exceeding $d$. If we think of $n$ as being chosen at random from a large initial segment of $\N$, then these events are asymptotically independent. It follows that the set $F_d = \{n \in \N \mid 0\not\in dR_n\}$ has natural density
\[\alpha_d = \P(d \text{ is not an } \N_0\text{-linear combination of } \{p \mid D_p = 1\}),\]
where $D_p \sim \mathrm{Bernoulli}(1/p)$ are independent random variables and $p$ runs over primes. In particular, for $d \geq 2$, the set $F_d$ is non-trivial in that $\alpha_d$ lies strictly between $0$ and $1$. To get a more precise estimate for $\alpha_d$ as $d \to \infty$, observe first the lower bound
\begin{equation}\label{eq:lower_bound}
    \alpha_d \geq \P\Biggl(\sum_{p \leq d}D_p = 1,\ D_p = 0 \text{ whenever } p \mid d\Biggr).
\end{equation}
To get a matching upper bound, one can use the fact that for coprime $p,q \in \N$, any integer not less than $(p-1)(q-1)$ can be represented as a non-negative integer linear combination of $p$ and $q$ (see e.g.\ \cite[Lemma 5.1]{lam-leung}). It follows that
\begin{equation}\label{eq:upper_bound}
    \alpha_d \leq \P(\not\exists p, p' \text{ such that } p<p',\ pp' \leq d \text{ and } D_p = D_{p'} = 1).
\end{equation}
Using \eqref{eq:lower_bound}, \eqref{eq:upper_bound} and Mertens' theorems \cite{mertens}, it is a routine exercise in analytic number theory to obtain the asymptotic
\[\alpha_d \sim \frac{e^{-\gamma}\log\log d}{\log d}.\]
\end{remark}

One may ask how the constant $c_{m,d}$ behaves as $d$ grows. The following gives a fairly precise answer to this question, at least in the most interesting case $m = \infty$.

\begin{proposition}
\label{prop:clt}
We have 
\[\lim_{d\to\infty}\Bigl(c_{\infty,d}-\frac{\log d}{2}\Bigr) = -\frac{\gamma}{2}.\]
\end{proposition}
\begin{proof}[Proof.]
    We will be somewhat brief on the details. Let $(Z_d)_{d\geq1}$ be a sequence of independent random variables with distribution $\eta_{\infty}$. For each $d \geq 1$ consider the random variable
    \[Y_d \vcentcolon= \frac{Z_1+\ldots+Z_d}{\sqrt{d}}.\]
    By the central limit theorem, $Y_d \to Z$ in distribution, where $Z$ is a standard complex Gaussian random variable. Since
    \[c_{\infty,d}-\frac{\log d}{2} = \E\bigl[\log|Y_d|\bigr]\]
    and we have
    \[\E\bigl[\log|Z|\bigr] = \frac{1}{\pi}\int_{\C}e^{-|z|^2}\log|z|\,dz = \frac{1}{2}\int_{0}^{\infty}e^{-s}\log s\,ds = -\frac{\gamma}{2},\]
    the conclusion will follow by Corollary \ref{cor:unif_int} once we show that the random variables $(\log|Y_d|)_{d\geq1}$ are uniformly integrable. The upper tail is not a problem since $(|Y_d|)_{d\geq1}$ are uniformly subgaussian. Specifically, Hoeffding's inequality implies the existence of constants $c,C>0$ such that
    \[\P(|Y_d|\geq t) \leq C\exp(-ct^2)\]
    for all $t > 0$ and $d \geq 1$, so it follows from \eqref{eq:int_of_log} that
    \[\lim_{t\to\infty}\sup_{d\geq1}\E\bigl[(\log|Y_d|)1_{\{\log|Y_d|\geq t\}}\bigr] = 0.\]
    To deal with the lower tail, we employ a Fourier argument in the style of \cite[\S7.3]{tao-vu-additive}. By standard estimates for Bessel functions (see \cite[Appendix B]{grafakos}), there is a constant $C > 0$ such that
    \begin{equation}\label{eq:osc_integral}
        |\widehat{\eta_{\infty}}(\xi)| \leq C|\xi|^{-1/2}
    \end{equation}
    for $\xi \in \C$. In particular, if $d \geq 5$, then $\reallywidehat{\eta_{\infty}^{*d}} = \widehat{\eta_{\infty}}^d \in \mathbb{L}^1(\C)$ (where we identify $\C \cong \R^2$), so by Fourier inversion, $\eta_{\infty}^{*d}$ has density
    \[f_d(z) = \int_{\C}\widehat{\eta_{\infty}}(\xi)^de\bigl(\mathrm{Re}(z\overline{\xi})\bigr)\,d\xi\]
    with respect to $\lambda$. Therefore, for any $r > 0$, we have
    \[\P(|Y_d| \leq r) = \int_{\overline{B}(0,r\sqrt{d})}f_d(z)\,dz \leq r^2d\pi\lVert f_d\rVert_{\infty} \ll d\lVert \widehat{\eta_{\infty}}\rVert_d^dr^2.\]
    Hence, the identity \eqref{eq:int_of_log} gives
    \[\E\Bigl[\log\Bigl(\frac{1}{|Y_d|}\Bigr)1_{\{|Y_d|\leq 1/T\}}\Bigr] = \P(|Y_d|\leq 1/T)\log T + \int_{T}^{\infty}\frac{\P(|Y_d|\leq 1/u)}{u}\,du \ll \frac{d\lVert \widehat{\eta_{\infty}}\rVert_d^d\log T}{T^2}\]
    for large $T$. It remains to show that
    \[\limsup_{d\to\infty}d\lVert \widehat{\eta_{\infty}}\rVert_d^d < \infty.\]
    Switching to polar coordinates and using the rotational invariance of $\widehat{\eta_{\infty}}$, we get
    \[d\lVert\widehat{\eta_{\infty}}\rVert_d^d = 2\pi d\int_{0}^{\infty}r|\widehat{\eta_{\infty}}(r)|^d\,dr.\]
    We decompose the range of integration into three regimes of values of $r$: small, intermediate and large. For large values, note that \eqref{eq:osc_integral} gives
    \[d\int_{C^2}^{\infty}r|\widehat{\eta_{\infty}}(r)|^d\,dr \leq dC^d\int_{C^2}^{\infty}r^{1-d/2}\,dr = \frac{2C^4d}{d-4},\]
    which is bounded as $d \to \infty$. To deal with the intermediate range, observe that
    \[d\int_{\delta}^{C^2}r|\widehat{\eta_{\infty}}(r)|^d\,dr \leq C^4d\Bigl(\max_{r\in[\delta,C^2]}|\widehat{\eta_{\infty}}(r)|\Bigr)^d,\]
    which goes to zero as $d \to \infty$. Indeed, the above maximum is less than $1$ since $|\widehat{\eta_{\infty}}(r)| < 1$ for all $r > 0$. Finally, to address the small values, note that if $\delta > 0$ is chosen to be sufficiently small, then the standard series expansion for Bessel functions (see \cite[Appendix B]{grafakos}) implies that $|\widehat{\eta_{\infty}}(r)| \leq 1-cr^2$ for all $0 \leq r \leq \delta$, where $c > 0$ is an absolute constant. Hence, we have
    \[d\int_{0}^{\delta}r|\widehat{\eta_{\infty}}(r)|^d\,dr \leq \int_{0}^{\delta}dr(1-cr^2)^{d-1}\,dr = \Bigl[-\frac{(1-cr^2)^d}{2c}\Bigr]_0^{\delta} <\frac{1}{2c},\]
    which completes the proof.
\end{proof}

\section{Concluding remarks}\label{sec:conc_rem}

In addition to the way we stated it, Problem \ref{prob:digraphs} is often phrased in terms of \emph{d-regular digraphs}. To see why, note that $A_n$ can be viewed as the adjacency matrix of a random $d$-regular directed graph on $n$ vertices, where we allow self-loops, but not multiple edges. Here, `$d$-regular' means that each vertex has both in- and out-degree equal to $d$. From this perspective, the matrices $C_n$ from our model are adjacency matrices of \emph{random $d$-regular Cayley graphs}. Due to their more explicit nature, random Cayley graphs have attracted a lot attention in Ramsey theory as a kind of structured proxy for random graphs. Particularly well-studied are the clique/chromatic number \cite{green-cayley, green-chromatic, mrazovic, campos-dahia-marciano, conlon-fox-pham-yepremyan}, diameter \cite{amir-gurel-gurevich, christofides-markstrom, marklof-strombergsson, helfgott-seress-zuk, eberhard-jezernik} and spectral/expansion properties \cite{alon-roichman, bourgain-gamburd} of random Cayley graphs. It is impossible for us to give a complete account of open problems in this vast area here. We only mention the important question of whether constant-degree random Cayley graphs of finite simple groups form a family of expanders, or equivalently possess a spectral gap (see e.g.\ \cite[Problem 79]{green-open-problems}).

We conclude this section by proposing the following problem, which may be worth considering as a kind of intermediate step towards Problem \ref{prob:digraphs}.

\begin{problem}
\label{prob:cayley}
What can be said about the limiting behaviour of ESDs of sparse random circulant matrices with respect to non-abelian groups? Specifically, let $C_n$ be a random $S_n$-circulant matrix with entries in $\{0,1\}$ and exactly $d$ ones in each row/column, where $S_n$ is the symmetric group. Does the sequence $(\mu_{C_n})_{n\geq1}$ converge weakly in probability?
\end{problem}

Problem \ref{prob:cayley} has some features in common with both Problem \ref{prob:digraphs} and our model. On the one hand, it retains both the sparsity and the non-commutativity of the matrices in question, which are crucial reasons why Problem \ref{prob:digraphs} is difficult. On the other hand, it possesses algebraic structure which is broadly similar to the model considered in our paper. It may be possible to take advantage of this structure, though it seems to us that this would require significant new ideas in the non-abelian case.

\bigskip

\noindent\textbf{Acknowledgements.} This work was supported by the Croatian Science Foundation under the project number HRZZ-IP-2022-10-5116 (FANAP). The author is grateful to Rudi Mrazović for encouragement and feedback on an earlier draft of this paper. He would also like to thank Sean Eberhard for several helpful comments and suggestions.

\appendix

\section{Some technical results}\label{app:tech_lemmas}

We collect here several auxiliary results of a technical nature. The first one asserts that the density of solutions to a system of linear equations over a finite abelian group equals that of the transposed system. This fact is presumably well-known, though we were not able to find a reference for it. Hence, we provide a short proof, which is based on Fourier analysis.

\begin{lemma}
\label{lm:lin_sys}
Let $G$ be a finite abelian group and let $r,s$ be positive integers. Let $A$ be an $r \times s$ matrix with integer entries. Then
\[\P_{x\in G^s}\bigl(Ax = 0\bigr) = \P_{y\in G^r}\bigl(A^Ty = 0\bigr).\]
\end{lemma}
\begin{proof}
    For each $j \in [s]$, consider 
    \[w_j \vcentcolon= \sum_{x\in G}1_{\{(A_{i,j}x)_{i\in[r]}\}}\]
    as a function on $G^r$. Then note that
    \[\P_{x\in G^s}\bigl(Ax = 0\bigr) = \frac{1}{|G|^s}(w_1*\ldots*w_s)(0).\]
    By the Fourier inversion formula and the fact that the Fourier transform diagonalises convolution, we have
    \[(w_1*\ldots*w_s)(0) = \frac{1}{|G^r|}\sum_{\gamma\in\widehat{G^r}}\prod_{j=1}^{s}\widehat{w_j}(\gamma).\]
    But note that, on writing $\gamma = \gamma_1 \otimes \ldots \otimes \gamma_r$ with $\gamma_i \in \widehat{G}$ for $i \in [r]$,
    \[\widehat{w_j}(\gamma) = \sum_{x \in G}\reallywidehat{1_{\{(A_{i,j}x)_{i\in[r]}\}}}(\gamma) = \sum_{x\in G}\prod_{i=1}^{r}\overline{\gamma_i(A_{i,j}x)} = \sum_{x \in G}\Biggl(\prod_{i=1}^{r}\gamma_i^{-A_{i,j}}\Biggr)(x).\]
    Hence, by orthogonality of characters, it follows that
    \[\widehat{w_j}(\gamma) = \begin{cases}|G| & \text{if } \prod_{i=1}^{r}\gamma_i^{A_{i,j}} = 1\\0 & \text{otherwise}\end{cases}.\]
    We obtain that
    \[\P_{x\in G^s}\bigl(Ax = 0\bigr) = \frac{1}{|G^r|}\Bigl|\Bigl\{\gamma \in \widehat{G^r}\ \Big|\ \forall j \in [s]\ \prod_{i=1}^{r}\gamma_i^{A_{i,j}} = 1 \Bigr\}\Bigr|,\]
    whence the conclusion follows by invoking the isomorphism $\widehat{G^r} \cong G^r$.
\end{proof}

\begin{remark}
\label{rem:adjoint}
Lemma \ref{lm:lin_sys} can also be proved by exploiting duality more directly, i.e.\ by considering the adjoint 
\[\alpha^* \colon \reallywidehat{G^r} \to \widehat{G^s}, \quad \gamma \mapsto \gamma \circ \alpha\] 
of the homomorphism $\alpha \colon G^s \to G^r$ induced by $A$.
\end{remark}

The following is a version of the method of moments \cite[Theorem 3.9]{fleermann-kirsch} adapted to our needs. Specifically, it says that for probability measures on $[0,\infty)$, it suffices to consider even moments. 

\begin{proposition}
\label{prop:moment_method}
For each $n \geq 1$ let $M_n$ be a random complex matrix. Suppose that there exists a sequence of non-negative real numbers $(b_k)_{k\geq1}$ such that:
\begin{enumerate}[(i)]
    \item for each $k \geq 1$ we have 
    \[\lim_{n\to\infty}\E\Bigl[\int_{[0,\infty)} t^{2k} \,d\nu_{M_n}(t)\Bigr] = b_k, \quad \lim_{n\to\infty}\mathrm{Var}\Bigl(\int_{[0,\infty)} t^{2k} \,d\nu_{M_n}(t)\Bigr) = 0;\]
    \item $\limsup_{k\to\infty}b_k^{1/(2k)}/(2k)<\infty$.
\end{enumerate}
Then there exists a probability measure $\nu$ on $[0,\infty)$ such that $\nu_{M_n} \to \nu$ weakly in probability.
\end{proposition}
\begin{proof}
    For each $n \geq 1$ consider the following random probability measure on $\R$:
    \[\widetilde{\nu}_n = \frac{1}{2r_n}\sum_{k=1}^{r_n}(\delta_{\sigma_k(M_n)}+\delta_{-\sigma_k(M_n)}),\]
    where $r_n$ denotes the dimension of $M_n$. If $k \geq 1$ is even, then
    \[\int_{\R} t^k\,d\widetilde{\nu}_n(t) = \int_{[0,\infty)} t^k\,d\nu_{M_n}(x),\]
    whereas if $k \geq 1$ is odd, then
    \[\int_{\R} t^k\,d\widetilde{\nu}_n(t) = 0.\]
    Therefore, by \cite[Theorem 3.3 (ii), Theorem 3.9 (ii)]{fleermann-kirsch}, we obtain a probability measure $\mu$ on $\R$ such that $\widetilde{\nu}_n \to \mu$ weakly in probability. Since $\nu_{M_n} = |\cdot|_*\widetilde{\nu}_n$, the conclusion follows with $\nu = |\cdot|_*\mu$.
\end{proof}

\bibliographystyle{plain}
\bibliography{references}

\end{document}